\documentclass[11pt]{amsart}

\usepackage[left=0.80in,right=0.80in,top=0.68in,bottom=0.62in]{geometry}
\usepackage{setspace}
\usepackage{microtype}
\usepackage{tikz}
\usepackage{float}
\usepackage{enumitem}

\allowdisplaybreaks

\usepackage{amsmath,amssymb,amsthm,mathtools}
\usepackage{aliascnt}

\usepackage{hyperref}
\hypersetup{colorlinks=true, linkcolor=blue, filecolor=magenta, urlcolor=cyan,  pdfpagemode=FullScreen,}

\usepackage[nameinlink,noabbrev, capitalize]{cleveref}

\newtheorem{theorem}{Theorem}[section]

\newaliascnt{proposition}{theorem}

\aliascntresetthe{proposition}

\newaliascnt{lemma}{theorem}
\newtheorem{lemma}[lemma]{Lemma}
\aliascntresetthe{lemma}

\newaliascnt{corollary}{theorem}

\aliascntresetthe{corollary}

\theoremstyle{definition}
\newaliascnt{definition}{theorem}

\aliascntresetthe{definition}

\crefname{theorem}{theorem}{theorems}
\Crefname{theorem}{Theorem}{Theorems}
\crefname{proposition}{proposition}{propositions}
\Crefname{proposition}{Proposition}{Propositions}
\crefname{lemma}{lemma}{lemmas}
\Crefname{lemma}{Lemma}{Lemmas}
\crefname{corollary}{corollary}{corollaries}
\Crefname{corollary}{Corollary}{Corollaries}
\crefname{definition}{definition}{definitions}
\Crefname{definition}{Definition}{Definitions}
\crefname{equation}{equation}{equations}
\Crefname{equation}{Equation}{Equations}
\crefname{section}{section}{sections}
\Crefname{section}{Section}{Sections}

\newlist{properties}{enumerate}{1}
\setlist[properties]{
    label=(P\arabic*),
    ref=P\arabic*
}

\crefname{propertiesi}{Property}{Properties}
\Crefname{propertiesi}{Property}{Properties}

\newcommand{\floor}[1]{\left\lfloor #1 \right\rfloor}
\newcommand{\ceil}[1]{\left\lceil #1 \right\rceil}

\usepackage{forest}
\usepackage{xparse}

\forestset{
  sequence tree/.style={
    for tree={
      inner sep=2pt,
      l sep=10mm,
      s sep=5mm,
      anchor=center,
      math content
    }
  }
}

\ExplSyntaxOn

\int_new:N \l__seqtree_none_int
\int_new:N \l__seqtree_ntwo_int
\int_new:N \l__seqtree_nthree_int

\int_new:N \l__seqtree_capone_int
\int_new:N \l__seqtree_captwo_int
\int_new:N \l__seqtree_capthree_int

\int_new:N \l__seqtree_L_int

\tl_new:N \l__seqtree_root_tl
\tl_new:N \l__seqtree_tree_tl

\prg_new_conditional:Npnn
  \__seqtree_can_extend:nnn #1#2#3
  { T, F, TF }
{
  \bool_if:nTF
  {
    \int_compare_p:n
      { #1 <= \l__seqtree_capone_int }
    &&
    \int_compare_p:n
      { #2 <= \l__seqtree_captwo_int }
    &&
    \int_compare_p:n
      { #3 <= \l__seqtree_capthree_int }
    &&
    (
      (
        \int_compare_p:n
          { \l__seqtree_L_int = 1 }
        &&
        \int_compare_p:n
          { #1 < \l__seqtree_none_int }
      )
      ||
      (
        \int_compare_p:n
          { \l__seqtree_L_int = 2 }
        &&
        \int_compare_p:n
          { #2 < \l__seqtree_ntwo_int }
      )
      ||
      (
        \int_compare_p:n
          { \l__seqtree_L_int = 3 }
        &&
        \int_compare_p:n
          { #3 < \l__seqtree_nthree_int }
      )
    )
  }
  { \prg_return_true: }
  { \prg_return_false: }
}

\cs_new_protected:Npn
  \__seqtree_build_node:nnnnn #1#2#3#4#5
{
  \tl_put_right:Nn \l__seqtree_tree_tl
  {
    [ { \mathtt{#1} }
  }

  \tl_if_blank:nF {#5}
  {
    \tl_put_right:Nn
      \l__seqtree_tree_tl
      { , #5 }
  }

  \__seqtree_can_extend:nnnT {#2}{#3}{#4}
  {

    \int_compare:nNnT
      {#2} < {\l__seqtree_capone_int}
    {
      \int_compare:nNnTF
        {\l__seqtree_L_int} = {1}
      {
        \__seqtree_build_node:nnnnn
          {#1 1}
          {#2+1}
          {#3}
          {#4}
          {edge={red, thick}}
      }
      {
        \__seqtree_build_node:nnnnn
          {#1 1}
          {#2+1}
          {#3}
          {#4}
          {edge={blue, thick}}
      }
    }

    \int_compare:nNnT
      {#3} < {\l__seqtree_captwo_int}
    {
      \int_compare:nNnTF
        {\l__seqtree_L_int} = {2}
      {
        \__seqtree_build_node:nnnnn
          {#1 2}
          {#2}
          {#3+1}
          {#4}
          {edge={red, thick}}
      }
      {
        \__seqtree_build_node:nnnnn
          {#1 2}
          {#2}
          {#3+1}
          {#4}
          {edge={blue,thick}}
      }
    }

    \int_compare:nNnT
      {#4} < {\l__seqtree_capthree_int}
    {
      \int_compare:nNnTF
        {\l__seqtree_L_int} = {3}
      {
        \__seqtree_build_node:nnnnn
          {#1 3}
          {#2}
          {#3}
          {#4+1}
          {edge={red, thick}}
      }
      {
        \__seqtree_build_node:nnnnn
          {#1 3}
          {#2}
          {#3}
          {#4+1}
          {edge={blue, thick}}
      }
    }
  }

  \tl_put_right:Nn
    \l__seqtree_tree_tl
    { ] }
}

\cs_generate_variant:Nn
  \__seqtree_build_node:nnnnn
  { Vnnnn }

\NewDocumentCommand{\SequenceTree}{mmm}
{
  \group_begin:

  \int_set:Nn
    \l__seqtree_L_int
    {#3}

  \int_set:Nn \l__seqtree_none_int
  {
    \int_div_truncate:nn {#1+2}{3}
  }

  \int_set:Nn \l__seqtree_ntwo_int
  {
    \int_div_truncate:nn {#1+1}{3}
  }

  \int_set:Nn \l__seqtree_nthree_int
  {
    \int_div_truncate:nn {#1}{3}
  }

  \int_set:Nn \l__seqtree_capone_int
  {
    \l__seqtree_none_int - 1
  }

  \int_set:Nn \l__seqtree_captwo_int
  {
    \l__seqtree_ntwo_int - 1
  }

  \int_set:Nn \l__seqtree_capthree_int
  {
    \l__seqtree_nthree_int - 1
  }

  \int_case:nnF {#3}
  {
    {1}
    {
      \int_set_eq:NN
        \l__seqtree_capone_int
        \l__seqtree_none_int
    }

    {2}
    {
      \int_set_eq:NN
        \l__seqtree_captwo_int
        \l__seqtree_ntwo_int
    }

    {3}
    {
      \int_set_eq:NN
        \l__seqtree_capthree_int
        \l__seqtree_nthree_int
    }
  }
  {
    \PackageError
      {SequenceTree}
      {L must be 1, 2, or 3}
      {}
  }

  \tl_set:Nx \l__seqtree_root_tl
  {
    \prg_replicate:nn {#2}{1}
  }

  \tl_clear:N \l__seqtree_tree_tl

  \__seqtree_build_node:Vnnnn
    \l__seqtree_root_tl
    {#2}
    {0}
    {0}
    {sequence~tree}

  \exp_args:NV
    \Forest
    \l__seqtree_tree_tl

  \group_end:
}

\ExplSyntaxOff

\newcommand{\R}{\mathbb{R}}

\title{A base-$8$ upper bound for planar peeling sequences}
\author{Andr\'e Hisatsuga}
\author{Griffin Johnston}
\author{Rafael Miyazaki}
\address{Department of Mathematics, Emory University, Atlanta, GA 30322, USA}
\email{\{andre.hisatsuga|john.johnston|rafael.kazuhiro.miyazaki\}@emory.edu}
\date{}

\begin{document}

\begin{abstract}
Let $g(n)$ denote the minimum number of peeling sequences among all $n$-point sets in general position in the plane. Dumitrescu and T\'oth proved an exponential upper bound with base $12.29$, and Simon subsequently lowered the base to $9.78$. Using the same recursive construction, we prove
\begin{equation*}
   g(n) \le (8+o(1))^n.
\end{equation*}

\end{abstract}

\maketitle

\section{Introduction}\label{sec:introduction}
Let $P\subset\R^2$ be a finite set in general position. A \emph{peeling sequence} of $P$ is an ordering of its points obtained by repeatedly deleting a vertex of the convex hull of the points that remain. We write $g(P)$ for the number of peeling sequences of $P$.
Observe that if $P$ is in convex position 
then $g(P) = n!$. Thus we are interested in 
minimizing the number of peeling sequences 
over all $n$-point sets $P$ in the plane and write 
\begin{equation*}
   g(n)=\min\{g(P): P\subset\R^2,\ |P|=n,\ P\text{ in general position}\}.
\end{equation*}
\par 
Peeling sequences were first introduced by 
Dumitrescu in \cite{dumitrescuOriginal2022peeling} where the 
bounds 
\begin{equation*}
    \Omega(3^n) \leq g(n) \leq 2^{n\log\log n}, 
\end{equation*}
were obtained. The lower bound follows 
from the simple observation that, while at least three points remain, there are at least 
three points on the convex hull of 
the remaining point set. 
The upper bound was subsequently improved 
by Dumitrescu and T\'oth in 
\cite{dumitrescuToth2025peeling} to 
\begin{equation*}
    g(n) \leq O(12.29^n), 
\end{equation*}
via a recursive construction, which we will 
describe in depth in \Cref{sec:construction} below. 

We remark that this construction is reminiscent of other recursive 
constructions in discrete geometry, including 
Edelsbrunner and Welzl's construction of point sets with $\Omega(n\log n)$ 
halving lines \cite{edelsbrunner1985number}, and the lower bound constructions for the Ramsey number 
of $n$-segments in the plane
\cite{LarMatPachToro, KaroPachToth97, kyncl2012ramsey}.  
\par 
By a more careful 
analysis of the construction 
of Dumitrescu and T\'oth, Simon \cite{simon2026further} obtained the bound 
\begin{equation*}
    g(n) \leq O(9.78^n).
\end{equation*}
Our main result is an improvement of 
the base of the exponent in the upper bound 
of $g(n)$ to $8+o(1)$. 
\par
There are two main novelties in our proof. 
Firstly, we  
obtain bounds for 
the peeling number for one-sided 
truncations of 
the recursive construction of 
Dumitrescu and T\'oth. These slightly more general objects allow our inductive argument to work.  
Secondly, 
the analysis of the peeling number of these 
truncations is simplified 
by applying a weighted version of 
the Kraft-McMillan inequality 
\cite{kraft1949device,mcmillan1956two} 
for prefix-free codes. 

\begin{theorem}\label{thm:main}
We have, 
\begin{equation*}
   g(n)\le(8+o(1))^n.
\end{equation*}
\end{theorem}

\subsection{General tools used in the proof}
We use the following weighted version of Kraft's inequality. The proof is a modification of 
that found in \cite{cover1991elements}, for example. We include the short proof for completeness.
For a tree $T$ rooted at a vertex $r$ we write 
$v \to w$ if $w$ is a child of $v$. 
\begin{lemma}\label{lem:weighted-kraft}
Let $T$ be a finite rooted tree. 
For every edge $v\to w$ of $T$, assign a
nonnegative weight $p(v,w)$, and suppose that for every vertex $v$,
\begin{equation*}
    \sum_{w:\, v\to w} p(v,w)\le 1.
\end{equation*}

For a vertex $u\in V(T)$, let its mass be the product of the weights
of the edges along the unique path from the root to $u$, with the root
having mass $1$. Then the set of leaves of $T$ has total mass at most $1$.
\end{lemma}

\begin{proof}
Let $\mu(u)$ denote the mass of a vertex $u$. Thus $\mu(r)=1$ for the
root $r$, and if $w$ is a child of $v$, then
$\mu(w)=\mu(v)p(v,w)$.
Consequently, for every vertex $v$,
\begin{equation*}
      \sum_{w:\,v\to w}\mu(w)
    =
    \mu(v)\sum_{w:\,v\to w}p(v,w)
    \le \mu(v).
\end{equation*}

Starting at the root and successively
replacing each non-leaf vertex $v$ by its children in
$T$ cannot increase the total mass, by the inequality
above. After finitely many such replacements, the resulting set of
vertices is precisely the set of leaves of $T$. Hence
\begin{equation*}
    \sum_{u \text{ leaf of }T}\mu(u)\le \mu(r)=1. \qedhere
\end{equation*}
    
\end{proof}
We observe that the classical $q$-ary Kraft inequality is recovered by assigning weight
$1/q$ to every edge of a $q$-ary prefix tree: a codeword $w$ of length
$\ell_w$ then has mass $q^{-\ell_w}$, and hence
\begin{equation*}
     \sum_w q^{-\ell_w}\le 1.
\end{equation*}
   
We will also use the following 
lemma, which is essentially 
\cite[Lemma 2]{dumitrescuToth2025peeling} 
and appears in the below form in \cite[Lemma 3.4]{simon2026further}. 
\begin{lemma}[Dumitrescu-T\'oth]\label{lem:partition-bound}
Let $P\subset\mathbb R^d$ be a finite point set in general position,
and let
\begin{equation*}
    P=P_1\sqcup\cdots\sqcup P_s.
\end{equation*}
Then
\begin{equation*}
        g(P)
    \le
    \binom{|P|}{|P_1|,\ldots,|P_s|}
    \prod_{i=1}^s g(P_i)
    \le
    s^{|P|}\prod_{i=1}^s g(P_i).
\end{equation*}
\end{lemma}

\section{The flat recursive construction}\label{sec:construction}

\begin{figure}[h]
\centering
\begin{tikzpicture}[scale=.74,>=stealth]
  \coordinate (O) at (0,0);
  \draw[->] (O) -- (2.55,0);
  \draw[->] (O) -- (-1.35,-2.34);
  \draw[->] (O) -- (-1.35,2.34);

  \draw[line width=2.2pt] (1.10,0.07) -- (1.85,-0.05);
  \draw[line width=2.2pt] (-.54,-1.12) -- (-.91,-1.72);
  \draw[line width=2.2pt] (-.82,1.55) -- (-1.16,2.10);

  \node[above right] at (1.82,-0.01) {$3$};
  \node[left] at (-0.98,-1.76) {$2$};
  \node[left] at (-1.22,2.03) {$1$};

  \node[below] at (0,-2.62) {before flattening};

  \begin{scope}[xshift=6.6cm]
    \draw[->] (-2.65,0) -- (2.75,0) node[right] {$x$};

    \draw[thin,black!45] (-2.35,.36) -- (0,0) -- (2.45,.10);
    \draw[thin,black!45] (-2.35,-.36) -- (0,0);

    \draw[line width=2.2pt] (-2.20,.29) -- (-1.48,.20);
    \draw[line width=2.2pt] (-.82,-.12) -- (-.24,-.045);
    \draw[line width=2.2pt] (1.05,.045) -- (1.92,.015);

    \draw[dashed,black!45] (-1.15,-.62) -- (-1.15,.62);
    \draw[dashed,black!45] (.38,-.62) -- (.38,.62);

    \node[above] at (-1.83,.33) {$1$};
    \node[below] at (-.53,-.20) {$2$};
    \node[above] at (1.49,.10) {$3$};

    \node[below] at (0,-2.62) {after flattening};
  \end{scope}
\end{tikzpicture}

\caption{One recursive step before and after flattening. The top-level blocks are labeled 1, 2, 3 according to their left-to-right order after flattening. Their $x$-projections are
pairwise disjoint and occur in this order.}
\label{fig:construction}
\end{figure}
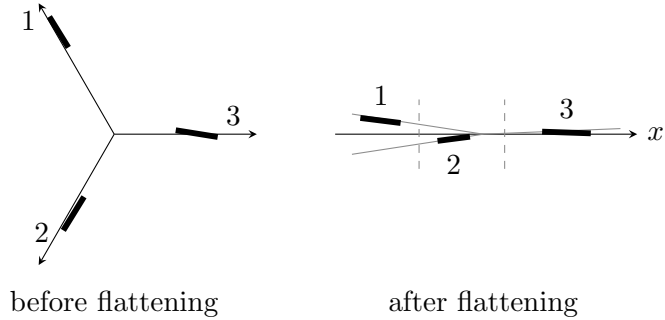

Our recursive construction of the sets 
$S_n$ is the one introduced by Dumitrescu and T\'oth in~\cite{dumitrescuToth2025peeling}. 
For completeness, we describe the construction 
below. 
$S_1$ will be a point, $S_2$ will be two points, 
and $S_3$ will be a flat obtuse triangle. 
For $n\geq 4$ we write 
$n = n_1 + n_2 + n_3$ where $n_1,n_2,n_3\in \{\floor{n/3}, \ceil{n/3}\}$. 
Let $r_1$, $r_2$, $r_3$ be three rays with arguments $ \frac{2\pi}{3}$, $\frac{4\pi}{3}$, $0$, respectively.
Place a copy of $S_{n_i}$ along the ray 
$r_i$ so 
that the line through any two points of 
$S_{n_i}$ has slope arbitrarily close to that of $r_i$. 
Call this copy of  
$S_{n_i}$ on the ray $r_i$ the \emph{$i$-th block} 
and denote it by 
$B_i$.
Translate the blocks along their corresponding rays and choose the copies sufficiently thin, as in~\cite{dumitrescuToth2025peeling,simon2026further}. Lastly, we flatten $S_n$ 
via an invertible 
affine transformation, choosing the parameters so that the $x$-projections of the three blocks 
are pairwise disjoint and occur in the order $B_1,B_2,B_3$ from left to right. That is, every point of $B_i$ lies to the left of every point of $B_{i+1}$. See \cref{fig:construction}.
\par 
An important feature of this construction 
is that the order in which the individual 
blocks 
are peeled is controlled in the initial 
stages of the peeling process. We 
make this notion precise below. 
\par 
An $a$-truncation of $S_n$ is the point set obtained from $S_n$ by either deleting its $a$ leftmost points or its $a$ rightmost points. We write these as $S_n^-[a]$ and $S_n^+[a]$ when orientation matters and define \begin{equation*}
    h(n,a)\coloneqq \max \{g(S_n^-[a]), g(S_n^+[a])\}.
\end{equation*}

For each block $B_i$ defined above, label the 
points in $B_i$ as 
$v_1^{(i)},\dots, v_{n_i}^{(i)}$ where the subscripts 
are ordered by decreasing distance from the origin 
along the ray $r_i$. So, for example, $v_1^{(i)}$ is 
furthest from the origin in the direction of 
$r_i$ and $v_{n_i}^{(i)}$ is closest to 
the origin. Deleting the first $a$ points of $B_i$ in this outer-to-inner order leaves an affine copy of one of the one-sided $a$-truncations of $S_{n_i}$ for $0\le a \le n_i$. Consequently its peeling number is at most $h(n_i,a)$.
\par 
At any step in some peeling process $\pi$ we say that a 
ray $r_i$ is \emph{active} if some points 
from the block $B_i$ have not yet been peeled. 
Furthermore, we call the first time at which some ray $r_{L(\pi)}$ ceases to be active the \emph{stopping time} of $\pi$.  We call $r_L = r_{L(\pi)}$ the \emph{leading ray} and the other two rays the \emph{surviving rays}.

As in the construction of Dumitrescu and T\'oth
\cite{dumitrescuToth2025peeling}, 
the parameters in the 
flattening and thinning steps can 
be chosen so that, 
while all three top-level
blocks are nonempty, the convex hull contains exactly one point
from each block, namely its current outermost point. Consequently, 
\begin{properties}
     \item\label{property1} Before stopping time only the outermost point 
    of each active block may be peeled.
\end{properties}

    So if, for example,
    $r_3$ is the leading ray, then at stopping time the surviving portions of $B_1$ and $B_2$ are  respectively one-sided $a$- and $b$-truncations of their recursive copies, for some $0\leq a < n_1$ and $0 \le b < n_2$. See \Cref{fig:surviving-structure}.

\begin{figure}[h]
\centering
\begin{tikzpicture}[scale=1.05,>=stealth]

\newcommand{\remainpt}[2]{%
  \fill (#1,#2) circle (1.8pt);
}
\newcommand{\deletedpt}[2]{%
  \draw[red, line width=.95pt]
    (#1-.06,#2-.06) -- (#1+.06,#2+.06)
    (#1-.06,#2+.06) -- (#1+.06,#2-.06);
}

\draw[->] (-3.0,0) -- (3.0,0) node[right] {$x$};
\draw[thin,black!40] (-2.55,.40) -- (0,0) -- (2.60,.11);
\draw[thin,black!40] (-2.55,-.40) -- (0,0);

\deletedpt{-2.20}{.345}
\deletedpt{-2.04}{.320}
\deletedpt{-1.88}{.295}
\remainpt{-1.72}{.270}
\remainpt{-1.56}{.245}
\remainpt{-1.40}{.220}
\node[above] at (-1.83,.43) {$1$};

\deletedpt{-.92}{-.145}
\deletedpt{-.78}{-.124}
\remainpt{-.64}{-.103}
\remainpt{-.50}{-.082}
\remainpt{-.36}{-.061}
\node[below] at (-.63,-.24) {$2$};

\deletedpt{1.10}{.047}
\deletedpt{1.28}{.040}
\deletedpt{1.46}{.033}
\deletedpt{1.64}{.026}
\deletedpt{1.82}{.019}
\node[above] at (1.47,.13) {$3$};

\draw[dashed,black!40] (-1.17,-.58) -- (-1.17,.58);
\draw[dashed,black!40] (.38,-.58) -- (.38,.58);

\end{tikzpicture}

\caption{
An example of a surviving structure at stopping time with leading ray $r_3$. Block $3$ has been exhausted, while blocks
$1$ and $2$ survive as one-sided truncations of their recursive copies.} 
\label{fig:surviving-structure}
\end{figure}
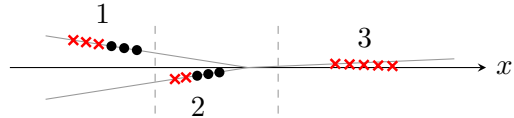

The separated $x$-projections in \Cref{fig:construction} also immediately imply
\begin{properties}
\setcounter{propertiesi}{1}
    \item \label{property2} Every one-sided truncation successively exhausts top-level blocks from the corresponding end. Thus the blocks already passed by the truncation are empty, at most one block is partially truncated, and all remaining blocks are untouched.
\end{properties}

\section{Combinatorial count}\label{sec:counting}
Given a peeling sequence $\pi$ of $S_n$, we 
may consider the word $\tilde{\pi}$ over the alphabet $[3]$ obtained 
from $\pi$ by replacing each point belonging to 
$B_i$ with the letter $i$. Following \cite{dumitrescuToth2025peeling} we call $\tilde{\pi}$ 
the \emph{simplified peeling sequence} of 
$\pi$. For each $i\in [3]$, let $x_i$ be the $i$-frequency function; that is, for each word $w$ over the alphabet $[3]$, let $x_i(w)$ be the number of times the letter $i$ occurs in $w$. Also, let $|w|$ denote the length of $w$.

A central part of our argument is to  
consider the prefixes of all 
possible simplified peeling sequences 
one may obtain given a fixed leading ray $L$ 
up to the point where $L$ ceases to be active. One can organize this information in a tree. 
\par 
Given a positive integer $n$, an index $L\in [3]$ and a non-negative integer $a<n_1$, the \emph{defended-prefix tree $T(n,a,L)$}, rooted at the string consisting 
of $a$ $1$'s,  
records all possible continuations 
of this initial simplified peeling sequence 
given that $L$ is the leading ray and $S_n$ has 
been $a$-truncated from the left, up to the 
point where $L$ is the first ray to become inactive. 
In particular, each vertex of $T$ 
corresponds to a 
word $w$ over the alphabet $\{1,2,3\}$. 
This word is a prefix of a 
simplified peeling sequence 
whose first $a$ letters are $1$'s and 
such that no block is exhausted before $B_{L}$ is exhausted. If $w$ has $n_{L}$ $L$'s then 
it is a leaf. Otherwise, 
the children of $w$ consist of all $1$-letter increments of $w$ consistent with the information that $L$ is the leading ray. Each leaf corresponds to the prefix of a simplified 
peeling sequence on $S_n$ in which $B_1$ is $a$-truncated in the initial steps, and $L$ is the first ray to become 
inactive. 
Furthermore, color an edge of $T(n,a,L)$ 
red if it corresponds to the deletion of a point from 
$L$. Otherwise, color the edge blue. See \cref{fig:T713}.
\par 
\begin{figure}[h]
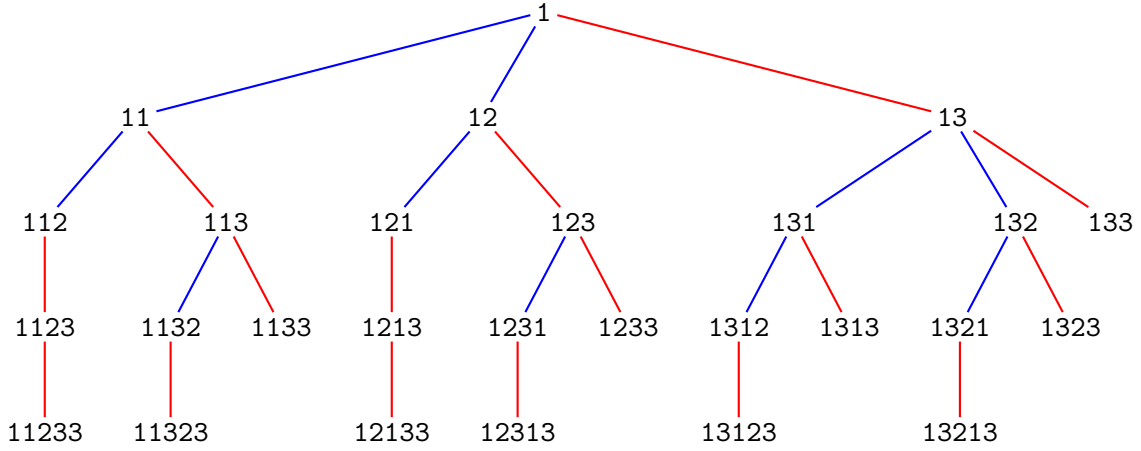

   \SequenceTree{7}{1}{3}
   \caption{The defended-prefix tree $T(7,1,3)$. Here $n_1=3, n_2=2,n_3=2$.}
   \label{fig:T713}
\end{figure}
\par 
Give every red edge a weight of $1/2$ and every blue 
edge a weight of $1/4$. Given a leaf 
$\ell$ in the leaf set $\mathcal{L}=\mathcal{L}(n,a,L)$ of $T(n,a,L)$, 
let $R(\ell)$ be the number of 
red edges in the unique path from the root of $T(n,a,L)$ 
to $\ell$ and $B(\ell)$ be the number of blue edges in this path. Every vertex has at most one red child and at most two blue children. Hence the total weight of its outgoing edges is at most $1/2 + 2\cdot 1/4= 1$, so \Cref{lem:weighted-kraft} applies. Thus for $\rho(\ell) \coloneqq 2^{-R(\ell)}4^{-B(\ell)}$, we have that the total leaf weight is at most $1$. That is
\begin{equation}\label{eqn:weighting}
    \sum_{\ell \in \mathcal{L}} \rho(\ell) \leq 1. 
\end{equation}

We will now use the above inequality to prove the following 
recursive upper bound, which 
will immediately imply \Cref{thm:main}.
\begin{lemma}\label{lemma:recursive}
    There exists a function $M(n)= e^{o(n)}$ such that for every integer $n\ge 1$ and every $0\le a\le n$ we have 
    \begin{equation*}
         h(n,a)\leq M(n)8^{n}2^{-a}\,.
    \end{equation*}
\end{lemma}
\begin{proof}
    Let $M(n) \coloneqq 
    12^{2^{\ceil{\log_3(n)}+1}-1} = e^{O(n^{\log_3 2})} = e^{o(n)}$. We will prove the required inequality by induction on $n$. For $n\leq 3$, every one-sided truncation contains at most three points, so its peeling number is at most $6$. Suppose now that $n\ge 4$ and assume that the required inequality holds for all $n'<n$. We shall prove, for all $0\le a\le n$, that 
    \begin{equation*}
        g(S_n^-[a]) \le M(n)8^n2^{-a}\,.
    \end{equation*}

    We first prove a property of the chosen function $M$ which we will use repeatedly. Let $k = \ceil{\log _3 n}$. Notice that for all $i\in [3]$, we have $n_i\le \ceil{n/3}$. Thus $\ceil{\log_3 n_i} \le k-1$ and $M(n_i) \le 12^{2^k-1}$. 
    
    This implies that for $1\le i\ne j \le 3$, one has  
    \begin{equation} \label{eq:M(n)_property}
        M(n_i)M(n_j) \le 12^{2^{k+1}-2} = \frac {M(n)}{12}.
    \end{equation}
    
    We first consider the case $a<n_1$, that is, when $S_n^-[a]$ still has three active rays. For a fixed leading ray $L$, consider the defended-prefix tree $T(n,a,L)$. The leaves $\ell\in\mathcal{L}(n,a,L)$ encode the simplified prefixes ending at the stopping time of peeling sequences that begin with the prescribed left $a$-truncation of $S_n$ and whose leading ray is $L$. By \cref{property1}, for each $i\in [3]\setminus \{L\}$, the portion of $B_i$ surviving at this stopping time is an $x_i(\ell)$-truncation of $S_{n_i}$. Hence
   
    \begin{equation}\label{eq: stopping_time_structure}
        g(S_n^-[a]) \le \sum_{L \in [3]} \sum_{\ell \in \mathcal{L}(n,a,L)} 2^{n-|\ell|} \prod_{i \in [3] \setminus \{L\}}h(n_i,x_i(\ell))\le \sum_{L \in [3]} \sum_{\ell \in \mathcal{L}(n,a,L)} 2^{n-|\ell|} \prod_{i\in [3]\setminus \{L\}} M(n_i)8^{n_i}2^{-x_i(\ell)},
    \end{equation}
    where the first inequality is given by \Cref{lem:partition-bound} and the second inequality is given by the induction hypothesis, since $n_i<n$ for all $i\in [3]$.

    For every $L\in [3]$ and $\ell \in \mathcal{L}(n,a,L)$, let 
    \begin{equation*}
        Q(\ell) \coloneqq 2^{n-|\ell|} \prod_{i\in [3]\setminus \{L\}} M(n_i)8^{n_i}2^{-x_i(\ell)}.
    \end{equation*}
    We next obtain a uniform upper bound for $Q(\ell)/\rho(\ell)$. Combined with \eqref{eqn:weighting} and \eqref{eq: stopping_time_structure}, this will give the desired estimate. For $\ell\in\mathcal L(n,a,L)$, set $x(\ell)\coloneqq \sum_{i\in[3]\setminus\{L\}} x_i(\ell)$.

    Note that 
    \begin{equation*}
        R(\ell) + B(\ell) = |\ell|-a,\qquad B(\ell)\le x(\ell).
    \end{equation*}
    Indeed, $R(\ell) + B(\ell)$ counts the number of point deletions after the initial $a$-truncation of $S_n$ up to the stopping time, at which point a total of $|\ell|$ points of $S_n$ have been deleted. Furthermore, 
    $B(\ell)$ counts the number of point deletions from blocks not in the leading ray after the initial $a$-truncation of $S_n$ up to the stopping time, while $x(\ell)$ counts the number of point deletions from blocks not in the leading ray up to the 
    stopping time. Thus, since $n-3n_L \le 2$, by \eqref{eq:M(n)_property} 
    \begin{equation}\label{eq: uniform_bounding}
         \frac{Q(\ell)}{\rho(\ell)}= \bigg(\prod_{i\ne L} M(n_i) \bigg)2^{|\ell|-a+B(\ell) + n - |\ell| +3(n-n_L) - x(\ell)}
         \le \frac{M(n)}{12}2^{3n_L+2 + 3(n-n_L)-a} = \frac{M(n)}{3}8^n2^{-a}. 
    \end{equation}
    Finally, combining \eqref{eq: stopping_time_structure}, \eqref{eq: uniform_bounding}, and \eqref{eqn:weighting} gives us
    \begin{equation*}
        g(S_n^-[a]) \le \sum_{L \in [3]} \sum_{\ell \in \mathcal{L}(n,a,L)} \rho(\ell) \frac{Q(\ell)}{\rho(\ell)} \le \sum_{L \in [3]} \frac{M(n)}{3}8^n2^{-a}\sum_{\ell \in \mathcal{L}(n,a,L)}\rho(\ell) \le  M(n)8^n2^{-a}.
    \end{equation*}
    Suppose now that $a \ge n_1$. By \cref{property2}, if $x_2 = \min\{n_2,a-n_1\}$ and $x_3 = a-n_1-x_2$ then $x_2$ and $x_3$ are respectively the number of points removed from $B_2$ and $B_3$ of an
    $a$-truncation of $S_n$. In particular, $x_2+x_3 = a-n_1 \ge 0$. Combining this inequality with \Cref{lem:partition-bound}, $n-3n_1\le 2$, \eqref{eq:M(n)_property}, and the 
    induction hypothesis
    gives 
    \begin{equation*}
        g(S_n^-[a]) \le 2^{n-a} M(n_2)8^{n_2}2^{-x_2}M(n_3)8^{n_3}2^{-x_3} \le \frac{M(n)}{12} 2^{3n +2-a} \le M(n)8^n2^{-a}.
    \end{equation*}

For right truncations the same argument applies with $B_1$ and $B_3$ interchanged. Indeed, the estimate $n-3n_i\le 2$ is uniform in $i$. We omit the details.
\end{proof}

\Cref{thm:main} now follows as a corollary of the above lemma.
\begin{proof}[Proof of \Cref{thm:main}]
 Notice that for $a=0$, an $a$-truncation of $S_n$ is $S_n$ itself. Thus, by \Cref{lemma:recursive}, there exists a subexponential function $M(n)$ such that
 \begin{equation*}
     g(n) \le g(S_n) = h(n,0) \le M(n)8^n = (8+o(1))^n.\qedhere
 \end{equation*} 
\end{proof}

\section{Concluding remarks}
For $P \subseteq \mathbb{R}^d$ in general position, 
let $g_d(P)$ be the number of peeling sequences of $P$ 
and let $g_d(n)$ be defined analogously. 
Dumitrescu and T\'oth in \cite{dumitrescuToth2025peeling} 
initiated the study of $g_d(P)$ for $d\geq 3$ 
with improved bounds found by Simon in \cite{simon2025minimum}. 
The same general framework developed here works to obtain upper bounds for $g_d(n)$ for $d\geq 3$. We very briefly describe how this strategy can be extended to higher-dimensional settings without giving full details for the sake of brevity. 
\par 
We use Simon's 
higher-dimensional recursive 
construction in \cite{simon2025minimum} on $n$ points, which 
we denote by $S_{n,d}$. 
Here we use Simon's notion 
of the \emph{defense of a point}. 
A point $p$ is said to be 
\emph{$m$-defended} by $S$ if $p$ is not peelable in 
the first $m$ steps of a peeling process on $S\cup \{p\}$. 
As in the planar case, we consider one-sided truncations 
of the segment-like copies of $S_{n,d}$. 
Instead of three rays in the planar case we consider 
$D = d+2m-1$ rays from the origin with 
copies of $S_{n_i,d}$ placed on the rays with  
$n_i \in \{\floor{n/D}, \ceil{n/D}\}$. The rays and placement 
of the copies of $S_{n_i,d}$ are arranged so that the origin is 
$m$-defended and the copies are horizontally separated. In particular, the higher-dimensional analog 
of \cref{property1} is maintained until $m$ blocks 
are exhausted. 
Given a peeling sequence $\pi$ of $S_{n,d}$, 
the \emph{stopping time} of $\pi$ is now the first 
time in which $m$ rays are no longer active. At a stopping 
time of $\pi$, we let $\Lambda$ denote the \emph{exhaustion ledger}, 
which records the order in which the blocks have been exhausted  
before the stopping time, which blocks have survived, 
and which new block is exposed after each block is exhausted. 
If we fix an $a$-truncation and an exhaustion ledger $\Lambda$,  
we can construct the higher-dimensional analog of the 
defended-prefix tree $T(n,a,\Lambda)$. 
The remaining computations are similar to the planar 
case and we omit the details. 
\subsection*{Acknowledgments and AI disclosure}
We thank Alexander Polyanskii for his helpful comments about the presentation of this paper. We used ChatGPT, running GPT-5.6 Sol (OpenAI), to assist with the development of the proof contained in this manuscript, with the preparation of the TikZ figures, and with writing the code used to generate the defended-prefix tree displayed in the manuscript. The final manuscript was written by the authors and we take full responsibility for all mathematical claims and for the final contents of the manuscript.
\bibliographystyle{plain}
\bibliography{bib}   

\end{document}